\documentclass[12pt,a4paper]{amsart}
\usepackage{tikz-cd}
\usepackage[T1]{fontenc}
\usepackage{amssymb,amsmath,latexsym,amsthm,paralist,setspace,a4,amssymb,mathrsfs,dsfont,MnSymbol,dsfont}

\usepackage{float}
\usetikzlibrary{babel}
\usepackage{varioref}
\usepackage{hyperref}
\usepackage[capitalise]{cleveref}
\newcommand{\Aut}{\mathrm{Aut}}

\newcommand{\Gal}{\mathrm{Gal}}

\newcommand{\m}{\mathfrak{m}}

\newcommand{\F}{\mathds F}

\newcommand{\C}{\mathds C}

\newcommand{\lang}{\longrightarrow}

\renewcommand{\O}{\mathcal{O}}

\newcommand{\cX}{{\mathscr X}}
\newcommand{\cY}{{\mathscr Y}}

\newcommand{\kiso}{\mathrel{\hbox{$\rightarrow$} \kern-1.5ex\lower-1ex\hbox{$\scriptstyle\sim$}\kern.5ex}}
\newcommand{\liso}{\mathrel{\hbox{$\longrightarrow$} \kern-2.4ex\lower-1ex\hbox{$\scriptstyle\sim$}\kern1.7ex}}

\newtheoremstyle{alexthm}
  {}
  {}
  {\sl }
  {}
  {\bf}
  {.}
  {.5em}
  {}
\theoremstyle{alexthm}
\newtheorem{itheorem}{Theorem}
\newtheorem{theorem}{Theorem}[section]
\newtheorem*{theorem*}{Theorem}
\newtheorem{corollary}[theorem]{Corollary}

\newtheorem{lemma}[theorem]{Lemma}
\newtheorem*{lemma*}{Lemma}

\newtheoremstyle{alexdef}
  {}
  {}
  {\rm }
  {}
  {\bf}
  {.}
  {.5em}
  {}
\theoremstyle{alexdef}
\newtheorem*{example*}{Example}
\newtheorem{example}[theorem]{Example}

\newtheorem{remark}[theorem]{Remark}

\newtheorem{definition}[theorem]{Definition}

\DeclareMathOperator{\Spec}{\textit{Spec}}

\DeclareMathOperator{\ch}{char}
\DeclareMathOperator{\Val}{Val}

\DeclareMathOperator{\Spa}{Spa}

\title{\bf On quasi-purity of the branch locus}
\author{Alexander Schmidt}
\address{Mathematisches  Institut, Universit\"{a}t Heidelberg, Im Neuenheimer Feld 205, 69120 Heidelberg, Deutschland}
\email{schmidt@mathi.uni-heidelberg.de}
\date{\today}

\begin{document}
\maketitle

The notion of ramification is classical and important in arithmetic geometry. There are essentially two different approaches: the valuation theoretic notion of ramification and the scheme (or ring) theoretic one. In case of an extension of Dedekind domains, both notions coincide since regular local rings of dimension one are exactly the valuation rings of discrete rank one valuations. In higher dimension or without regularity assumptions the two notions of ramification diverge.

Let $K/k$ be a finitely generated field extension. A $k$-valuation of $K$ is a valuation $v$ on $K$ which is trivial on $k$. We call a  normal, connected scheme $X/k$ separated and of finite type with function field $K$ a  \emph{model} of $K$. The normalization of $X$ in a finite, separable extension $L/K$ is denoted by $X_L$.
The main result of this paper is the following

\begin{itheorem} [Quasi-purity of the branch locus] \label{qp}
Let $L/K$ be a finite separable extension which is ramified at some $k$-valuation $w$ of $L$. Then there exists a model $X$ of $K$ and a (Weil) prime divisor $D\subset X_L$ which is ramified in the scheme morphism $X_L\to X$.
\end{itheorem}

Assuming the existence of a \emph{regular, proper} model $X$ of $K$, \cref{qp} is a straight-forward consequence of the Zariski-Nagata theorem on the purity of the branch locus. The existence of a regular, proper model of $K$ is known if
$\ch(k)=0$ \cite{Hi64}; if $\ch(k)=p>0$ and $\mathrm{tr.deg}_k K \le 2$ \cite{Li78}, and if
$\ch(k)=p>0$, $[k:k^p]< \infty$ and $\mathrm{tr.deg}_k K = 3$ \cite{CP08,CP09}.
In this paper we avoid assumptions on  resolution of singularities by using M.~Temkin's inseparable local uniformization theorem \cite{Te13} instead.

\medskip\noindent
As an application of \cref{qp} we show the following  \cref{tamethm}; see \cite{KS10} for the notion of curve-tameness and  \cite{Hue18} for that of  tameness for \'{e}tale morphisms of adic spaces. We recall the relevant definitions in Sections \ref{tamesec} and~\ref{spa-sec}.

\begin{itheorem}\label{tamethm}
Let $k$ be a field of positive characteristic, $X$ and $Y$ schemes, separated and of finite type over $k$ and $f: Y \to X$ a finite, \'{e}tale $k$-morphism.  Let
\[
\Spa(f):\ \Spa(Y,k) \lang \Spa(X,k)
\]
be the associated \'{e}tale morphism of adic spaces.

Then $f$ is curve-tame if and only if\/ $\Spa(f)$ is  tame.
\end{itheorem}

\section{Passage to the algebraic closure}
Let $K$ be a field (imperfect, otherwise the following discussion is void), $\bar K$ an algebraic closure of $K$, $K^s$ the separable closure of $K$ in $\bar K$ and $K_\infty= K^{1/p^\infty}$ the perfect closure of $K$ in $\bar K$. Then the natural map $\Aut(\bar K/K) \to \Aut(K^s/K)$, $\sigma\mapsto \sigma|_{K^s}$ induces an isomorphism
\[
\varphi: \Gal(\bar K/K_\infty) \liso \Gal(K^s/K).
\]
\begin{lemma} \label{ramfilt}
Let $v$ be a valuation on $\bar K$ and let $v^s$ be its restriction to $K^s$. Then $\varphi$ induces isomorphisms
\[
D_v \kiso D_{v^s}, \ I_v \kiso  I_{v^s}, \ R_v \kiso R_{v^s},\leqno (*)
\]
where the letters $D$, $I$ and $R$ denote the decomposition, inertia and ramification groups of the respective valuations.
\end{lemma}

\begin{proof}
Since $v$ is the only extension of $v^s$ to $\bar K$ \cite[Corollary 3.2.10]{EP05}, we have for $\sigma \in \Aut(\bar K/K)$ that
\[
\sigma v= v \ \Leftrightarrow \ \sigma|_{K^s}(v^s)=v^s,
\]
hence $\varphi$ induces an isomorphism $D_v \kiso D_{v^s}$.

Since $K_\infty/K$ is purely inseparable,  the same is true for the residue field extension
\[
\kappa(v|_{K_\infty})/\kappa(v|_K).
\]
Hence we obtain the commutative diagram
\[
\begin{tikzcd}
1\rar&I_v \rar&D_v\dar{\wr}\rar &G_{\kappa(v|_{K_\infty})}\dar{\wr}\rar&1\\
1\rar&I_{v^s} \rar&D_{v^s}\rar &G_{\kappa(v|_{K})}\rar&1
\end{tikzcd}
\]
inducing the claimed isomorphism $I_v \kiso I_{v^s}$. Finally, the ramification groups are the $p$-Sylow subgroups of the inertia groups, showing that also $R_v \kiso R_{v^s}$.
\end{proof}
The ramification indices are in general not preserved under inseparable base change:

\begin{example}
Let $K=\F_p(X,Y)$ and $L=K[T]/(T^p-XT-Y)$. The valuation of $K$ associated with $X$ does not split in $L/K$, the ramification index is equal to $1$, and the residue field extension is the inseparable extension $\F_p(Y^{1/p})/\F_p(Y)$. Now consider $K'=K(Y^{1/p})=\F_p(X,Y^{1/p})$, $L'=K'L$. Over $K'$ the polynomial $T^p-XT-Y$ can be written as $(T-Y^{1/p})^p-X(T-Y^{1/p}) + XY^{1/p}$. Up to a substitution, it is an Eisenstein polynomial. Hence the valuation of $K'$ associated with $X$ does not split in $L'/K'$, the ramification index is equal to $p$, and the residue extension is trivial.
\end{example}

\section{\'{E}tale versus unramified}

The following lemma seems to be well-known  but we could not find a reference.

\begin{lemma} \label{etale} Let $L/K$ be a finite field extension, $w$ a non-archimedean valuation of $L$ and $v=w|_K$. Then $w/v$ is unramified (in the valuation-theoretic sense) if and only if $\O_v\to O_w$ is \'{e}tale in the ring-theoretic sense.
\end{lemma}

\begin{proof}
Let $\O_{v,L}$ be the integral closure of $\O_v$ in $L$. Then, for every maximal ideal $\m\subset \O_{v,L}$, the localization $(\O_{v,L})_\m$ is a valuation ring of a valuation of $L$ which extends $v$ and, by \cite[Theorem 3.2.13]{EP05}, the assignment
\[
\m \longmapsto (\O_{v,L})_\m
\]
gives a 1:1-correspondence between the maximal ideals of $\O_{v,L}$ and those valuations. In particular, $\O_{v,L}$ is semi-local. As is well known, the inertia group of the action of $G=\Gal(L/K)$ detects ramification in the valuation-theoretic sense. The same is true for \'{e}tale by
\cite[Ch.\ X, Theorem 1]{Ra70}.
\end{proof}

\begin{example}[A.~Holschbach] \label{holshexample} Consider the field $K=\C((T))$ and let $L=K(T^{\frac{1}{2}})$. Consider
\[
K_\infty= \bigcup_{r=1}^\infty K(T^{1/3^r}), \ L_\infty=LK_\infty.
\]
Then $L_\infty/K_\infty$ is ramified in the valuation-theoretic sense. The associated rings of integers are
\[
A=\C[[T]][T^{1/3^r}, r\geq 1], \quad B= \C[[T]][T^{1/2\cdot 3^r}, r\ge 1].
\]
Hence the ring extension $B/A$ satisfies $\m_AB=\m_B$. But it is not of finite type, hence not \'{e}tale in the ring-theoretic sense.
\end{example}


\begin{lemma}\label{ramover}
Let $k$ be a field, $X$ a normal, connected and separated scheme of finite type over $k$, $K=k(X)$, $L/K$ a finite, separable field extension and $Y$ the normalization of $X$ in $L$. Let $w$ be a $k$-valuation on $L$ having center $y \in Y$ and let $v$ be the restriction of $w$ to $K$.  Assume that $Y\to X$ is \'{e}tale at $y$. Then $w/v$ is unramified.
\end{lemma}

\begin{proof} Let $s$ be the special point  of $\Spec(\O_w)$ and $t$ its image under $\Spec(\O_w) \to Y\times_X \Spec(\O_v)$. Since $t$ maps to $y$ in $Y$ and $Y\to X$ is \'{e}tale at $y$, the base change $Y\times_X \Spec(\O_v) \to \Spec(\O_v)$ is \'{e}tale at $t$. As \'{e}tale schemes over normal schemes are normal \cite[Ch.\ VII, Prop.\,2]{Ra70}, the local ring of $Y\times_X \Spec(\O_v)$ at $t$ is normal and hence isomorphic to $\O_w$ (cf.\ the proof of \cref{etale}). Since $Y\times_X \Spec(\O_v) \to \Spec(\O_v)$ is \'{e}tale at $t$, $\O_w/\O_v$ is \'{e}tale.
\end{proof}

\section{Quasi-Purity}

Let $K/k$ be a finitely generated field extension of transcendence degree $d$ and let $v$ be a discrete rank one $k$-valuation on $K$. By Abhyankar's inequality (cf., e.g., \cite[Prop.\,3.2]{KS10}), we have for the residue field $Kv$ of $v$ that
\[
\mathrm{deg.tr}_k  (Kv) \le d-1. \eqno (*)
\]
We call $v$ \emph{geometric}, if equality holds in $(*)$. By \cite[\S 8,\, Thm.\,3.26\,(b)]{Liu} $v$ is geometric if and only if there exists a model $X$ of $K$ and a Weil prime divisor $D\subset X$ with $\O_v=\O_{X,D}$. Hence \cref{qp} is equivalent to

\begin{theorem}\label{qp2}
Let $k$ be a field, $K/k$ a finitely generated field extension and $L/K$ a finite, separable  extension. Assume there exists a $k$-valuation $w$ on $L$ with restriction $v$ to $K$ such that $w/v$ is ramified. Then there exists a geometric discrete rank one $k$-valuation $W$ on $L$ with restriction $V$ to $K$ such that $W/V$ is ramified.
\end{theorem}

\begin{proof}
We choose a proper model $X$ of $K$. By M.~Temkin's inseparable local uniformization theorem \cite[Cor.\,1.3.3]{Te13}, we find a regular, connected $k$-scheme $X'$ and a $k$-morphism $X'\to X$ such that $K'=k(X')/K=k(X)$ is a finite, purely inseparable extension and the unique $k$-valuation $v'$ of $K'$ extending $v$ has center in $X'$. Let $L'=LK'$ be the composite (in some algebraic closure of $K$) and $w'$ the unique $k$-valuation of $L'$ lying over $w$. Then $w'|_{K'}=v'$ and $w'/v'$ is ramified by \cref{ramfilt}. Hence, by \cref{ramover}, the scheme morphism $X'_L \to X'$ is not \'{e}tale. By Zariski-Nagata purity \cite[X, Th\'{e}or\`{e}me 3.4]{SGA2}, we find a ramified divisor, hence a geometric rank one $k$-valuation $W'$ of $L'$ with restriction $V'$ to $K'$ such that $W'/V'$ is ramified. Denoting the respective restrictions to $L$ and $K$ by $W$ and $V$, another application of \cref{ramfilt} shows that $W/V$ is ramified.
\end{proof}

We will use the following mild sharpening of \cref{qp} in the proof of \cref{tamethm2} below.

\begin{corollary}\label{qp-sharp}
Under the assumptions of \cref{qp} let $U$ be a model of $K$ such that $U_L\to U$ is \'{e}tale. Then we can choose the model $X$ of $K$ asserted in \cref{qp} in such a way that there is an open $k$-immersion $j: U \hookrightarrow X$ and the ramified prime divisor $D$ is contained in $ X_L \smallsetminus U_L$.
\end{corollary}

\begin{proof}
We choose a proper model $\bar U$ of $K$ containing $U$ as an open subscheme. By \cref{qp2}, we find a geometric discrete rank one $k$-valuation $W$ of $L$ with restriction $V$ to $K$ such that $W/V$ is ramified. Since $\bar U$ is proper, $V$ has a nonempty center on $\bar U$ which is contained in $\bar U \smallsetminus U$ since $U_L\to U$ is \'{e}tale.  Following \cite[\S 8,\,Exercise 3.14]{Liu}, by successively blowing up $\bar U$ in centers contained in $\bar U \smallsetminus U$ and finally normalizing, we find a normal compactification $X$ of $U$ such that $V$ is the valuation associated to a point $x\in X\smallsetminus U$ of codimension one in $X$. This finishes the proof.
\end{proof}

\section{Curve-tameness}\label{tamesec}
Let $k$ be a field of positive characteristic. By \emph{variety} we mean a separated scheme of finite type over $k$,
a {\em curve}  $C$ is an integral variety with $\dim C=1$ and by \emph{\'{e}tale covering} we mean finite, \'{e}tale morphism. For a regular curve $C$ there exists a unique regular curve $P(C)$ which is proper and contains $C$ as a dense open subscheme.
Recall that an \'{e}tale covering $C'\to C$ of regular curves is called {\em tamely ramified along $P(C) \smallsetminus C$} if for every $x\in P(C) \smallsetminus C$ the associated valuation $v_x$ is tamely ramified in the finite, separable field extension $k(C')/k(C)$. This definition extends to the case of general regular varieties  of dimension one by requiring tameness on every connected component.

\medskip
Recall the following definitions from \cite{KS10}:

\begin{definition}
An \'{e}tale covering $Y\to X$ of varieties is {\em curve-tame} if for any morphism $C\to X$ with $C$ a regular curve, the base change $Y\times_X  C \to  C$ is tamely ramified along $P(C)\smallsetminus C$.

If, in addition, $X$ and $Y$ are normal and connected,  we say that $Y\to X$ is
{\em valuation-tame} if every $k$-valuation of $k(X)$ is tamely ramified in the field extension  $k(Y)/k(X)$.
This definition extends to coverings of general normal varieties by requiring valuation tameness on every connected component.
\end{definition}

By definition, the notions of curve- and valuation-tameness agree for coverings of regular curves. The statement of the next lemma follows directly from the definitions.

\begin{lemma}\label{compo}
Let $g: Z\to Y$ and $f:Y\to X$ be \'{e}tale coverings. If $g$ and $f$ are curve-tame, then the same holds for $f\circ g$. If $f\circ g$ is curve-tame, then $g$ is curve-tame and if, in addition, $g$ is surjective, then also $f$ is curve-tame.

The same holds for valuation-tame instead of curve-tame. \qed
\end{lemma}
\begin{lemma}\label{galclos}
\begin{enumerate}[\rm (i)]
  \item An \'{e}tale covering of connected varieties is curve-tame if and only if its Galois closure is curve-tame.
  \item An \'{e}tale covering of normal, connected varieties is valuation-tame if and only if its Galois closure is valuation-tame.
\end{enumerate}
\end{lemma}

\begin{proof} If the Galois closure  $\widetilde Y \to X$ of $Y \to X$ is curve-tame, then $Y\to X$ is curve-tame by \cref{compo}.
Directly from the definition we see that for curve-tame coverings $Y_1\to X$, $Y_2\to X$ also their fibre product $Y_1 \times_X  Y_2\to X$ is curve-tame. Now the Galois closure $\widetilde Y \to X$  of $Y\to X$ occurs as a connected component (of maximal degree) of the $d$-fold self product $Y\times_X \cdots \times_X Y$ with $d=\deg (Y/X)$. Hence $\widetilde Y \to X$ is curve-tame if $Y\to X$ is.

The same arguments apply for valuation-tameness.
\end{proof}

The main result of this section is

\begin{theorem}\label{tamethm2}
Let $k$ be a field of positive characteristic and let $f: Y \to X$ be a finite, \'{e}tale morphism of regular $k$-varieties. Then $f$ is curve-tame if and only if it is valuation-tame.
\end{theorem}

\begin{remark}
\Cref{tamethm2}
sharpens \cite[Theorem 4.4]{KS10} (which makes an assumption on the existence of regular, proper models) in the case that the base scheme is the spectrum of a field.
\end{remark}
\begin{proof}[Proof of \cref{tamethm2}]
By \cref{galclos}, we may assume that $f: Y\to X$ is a Galois covering of connected varieties.

Assume that there exists a $k$-valuation $w$ of $k(Y)$ which is wildly ramified in $k(Y)/k(X)$. Let $1\ne R_w\subset \Gal(Y/X)$ be the ramification group of $w$ and let $G\subset R_w$ be a cyclic subgroup of order $p$. Setting $Z=Y/G$, we obtain a cyclic Galois covering $Y/Z$ of order $p$ in which $w$ is (wildly) ramified. By \cref{qp-sharp}, we find a normal, connected variety $\bar{Z}$ containing $Z$ as a dense open subscheme and a prime divisor $D\subset \bar Z$ which ramifies in $\bar Z_{k(Y)}\to  \bar Z$. By the Key Lemma~2.4 of \cite{KS10}, we find a regular curve $C$ and a morphism $C\to Z$ such that the base change $Y\times_{Z}C \to C$ is (wildly) ramified along $P(C)$. Hence $Y\to Z$ is not curve-tame and by \cref{compo}, also $Y\to X$ is not curve-tame. This shows that curve-tameness implies valuation tameness.

The other implication is part of \cite[Theorem 4.4]{KS10}.
\end{proof}

\section{Tame morphisms of adic spaces}\label{spa-sec}
We refer the reader to \cite{Hu96} for basic notions on adic spaces. Following \cite[\S3]{Hue18}, we call an \'{e}tale morphism $\cY \to \cX$ of adic spaces  \emph{tame} if for every point $y\in \cY$ with image $x\in \cX$ the extension of the valuations associated to $y$ and $x$ is (at most) tamely ramified.

\medskip
In \cite{Te11} M.~Temkin associates with a morphism of schemes $X\to S$ a discretely ringed adic space $\Spa(X,S)$. If $X=\Spec(A)$ and $S=\Spec(R)$ are affine, then $\Spa(X,S)$ coincides with Huber's affinoid adic space $\Spa(A,A^+)$, where $A$ is endowed with the discrete topology and $A^+$ is the integral closure of the image of $R$ in $A$.

\medskip
If $S=\Spec(k)$ is the spectrum of a field and $X$ is a variety over $k$, then the underlying set of points of $\Spa(X,k)$ is the following:
\[
|\Spa(X,k)|= \coprod_{x\in X} \Val_k (k(x)) \quad\quad \text{(set-theoretically)}. 
\]
Here  $\Val_k (k(x))$ is the set of $k$-valuations of $k(x)$.

\medskip
An \'{e}tale morphism of $k$-varieties $Y\to X$ induces an \'{e}tale morphism of adic spaces $\Spa(Y,k)\to \Spa(X,k)$, cf.\  \cite{Hu96}. With this preparations  we are ready to prove \cref{tamethm}.

\begin{proof}[Proof of \cref{tamethm}]
By definition, $f$ is curve-tame if and only if for every point $y\in Y_1$ (i.e.\ the closure of $y$ is a curve) with image $x\in X$, the extension $k(y)/k(x)$ is tamely ramified at every  $v\in \Val_k k(y)$. In other words, the tameness of $\Spa(f)$ trivially implies the curve-tameness of $f$.

Conversely, assume that $f$ is curve-tame. Let $y\in Y$ be a point with image $x\in X$ and $v\in \Val_k k(y)$. We have to show that $v$ is tamely ramified in $k(y)/k(x)$. Replacing $X$ by the closure of $x$ and using that curve-tameness is stable under base change, we may assume that $x$ is the generic point of the integral variety~$X$. Furthermore, after replacing $X$ by an open subscheme, we may assume that $X$ (and hence also $Y$) is regular.
Then, by \cref{tamethm2}, $f$ is valuation-tame, hence $v$ is tamely ramified in $k(y)/k(x)$.
\end{proof}

\section*{Acknowledgement}
The author thanks A.~Holschbach for helpful discussions and for providing \cref{holshexample}. Moreover, we thank K.~H\"{u}bner for her comments on a preliminary version of this article.

\end{document}